\documentclass[12pt,a4paper]{amsart}

\usepackage[utf8]{inputenc}
\usepackage[english]{babel}

\usepackage{latexsym}
\usepackage{amsfonts}
\usepackage{amsmath}
\usepackage{amssymb}
\usepackage{graphicx}
\usepackage{color}
\usepackage{enumerate}

\usepackage{amsthm}

\newcommand{\R}{\mathbb{R}}
\newcommand{\N}{\mathbb{N}}

\newcommand{\C}{\mathbb{C}}

\newcommand{\ov}[1]{\overline{#1}}
\newcommand{\eps}{\epsilon}

\renewcommand\Re{\operatorname{Re}}
\renewcommand\Im{\operatorname{Im}}
\newcommand\spec{\operatorname{Spec}}

\newcommand\p{\operatorname{\partial}} 
\newcommand\vp{\varphi}

\newtheorem{thm}{Theorem}[section]
\newtheorem{cor}[thm]{Corollary}
\newtheorem{lem}[thm]{Lemma}
\newtheorem{prop}[thm]{Proposition}

\newtheorem{rem}[thm]{Remark}

\theoremstyle{definition}

\numberwithin{equation}{section}

\begin{document}

\title[]{Multidimensional Borg--Levinson theorems for unbounded potentials}
\date{}

\author[]{Valter Pohjola}
\address{Department of Mathematics and Statistics,  University of Jyväskylä}
\email{valter.pohjola@gmail.com}

\begin{abstract} 
We prove that the Dirichlet 
eigenvalues and Neumann boundary data of the corresponding eigenfunctions of the operator $-\Delta + q$,
determine the potential $q$, when $q \in L^{n/2}(\Omega,\R)$ and $n \geq 3$.
We also consider the case of incomplete spectral data, in the sense that the above spectral data 
is unknown for some finite number of eigenvalues. In this case we prove that the potential 
$q$ is uniquely determined for $q \in L^p(\Omega,\R)$ with $p=n/2$, for $n\geq4$ and
$p>n/2$, for $n=3$.

\end{abstract}

\maketitle

\section{Introduction}

Let $\Omega \subset \R^n$, $n \geq 3$ be a bounded domain, with a smooth boundary.
The operator $-\Delta + q$, with $q \in L^{n/2}(\Omega,\R)$ and form domain $H^1_0(\Omega)$,
has  a spectrum consisting of a discrete set of real eigenvalues, $\lambda_{k}$ of
finite mutliplicity, such that
\[
-\infty < \lambda_1 \leq \lambda_2 \leq \dots  \leq \lambda_k \to \infty,
\]
as $k\to \infty$.  The eigenvalues correspond to eigenfunctions
$\varphi_k$, which form an orthonormal basis in $L^2(\Omega)$
(see the appendix in section \ref{sec:spec} for some further discussion).

The  multi-dimensional Borg-Levinson problem first considered in \cite{NSU},
by Nachman, Sylvester and Uhlmann and independently by Novi\-kov in \cite{N},
is an inverse spectral problem that asks if a potential $q$ is uniquely determined
if one knows 
\begin{align}
    \lambda_k \text{  and  } \nu \cdot \nabla \varphi_k |_{\p\Omega},\; \text{ for } k \in \N,
\end{align}
where $\nu$ is the outward pointing unit normal vector to the boundary $\p\Omega$.
Nachman, Sylvester and Uhlmann showed that this is indeed possible for $q \in L^\infty(\Omega,\R)$.
This result is a higher dimensional variant of a question studied originally
by Borg in \cite{B} and Levinson
in \cite{L}, in the case of the 1-dimensional Schr\"{o}dinger equation.

\medskip
\noindent
The multi-dimensional Borg-Levinson problem has been studied in a number of settings.
It is not possible to give an extensive survey of this here and we will only mention a few 
results which are of relevance here.

The case unbounded or singular potentials $q \in L^p(\Omega,\R)$ $p>n/2$, has
been studied
by Päivärinta and Serov in \cite{PS}.
Krupchyk and Päivärinta studied the problem for higher order elliptic operators in \cite{KP},
in the case that $q \in L^\infty(\Omega,\R)$.

The problem has also been considered in the case when the spectral data for  a 
finite number of the eigenvalues is unknown.
One of the first to study the case of incomplete spectral data was Isozaki in \cite{I}.
A further interesting development in this direction is the work by Choulli and Stefanov in \cite{CS}
where they show that one only needs assume that the spectral data is asymptotically near to
each other, to obtain uniqueness.

\medskip
The main results here are the following. We use the notation $\lambda_{q_j,k}$ and $\vp_{q_j,k}$
for the $k$th eigenvalue and eigenvector corresponding to the operator $-\Delta+q_j$.
And the notation $\tilde \gamma u$ for the trace of the normal derivative of $u$ to $\p \Omega$, which corresponds 
to $\nu \cdot \nabla u |_{\p\Omega}$ when $u$ is smooth (see section \ref{sec:app2} for further details).

\begin{thm}\label{thm1}
Suppose that $q_1,q_2\in L^{n/2}(\Omega,\R)$, with $n\geq3$
and that 
\[
    \lambda_{q_1,k} = \lambda_{q_2,k}
    \quad\text{and}\quad 
    \quad \tilde\gamma \varphi_{q_1,k} = \tilde\gamma \varphi_{q_2,k}, 
\]
for $k \in \N$, then $q_1 = q_2$.
\end{thm}

In the case of incomplete spectral data we have.

\begin{thm}\label{thm2}
Let $q_1,q_2 \in L^p(\Omega,\R)$, where $p = n/2$, when $n \geq 4$ and
$p>n/2$, when $n=3$, and suppose that there is a $k_0 \in \N$ such that
\[
    \lambda_{q_1,k} = \lambda_{q_2,k}
    \quad\text{and}\quad 
    \quad \tilde\gamma \varphi_{q_1,k} = \tilde\gamma \varphi_{q_2,k}, 
\]
for $k \geq k_0$, then $q_1 = q_2$.
\end{thm}

\medskip
\noindent
The above Theorems improve the results in \cite{PS} in two ways when $n\geq3$.
We firstly prove the  Borg-Levinson Theorem for singular or unbounded potentials in the 
limiting case of $q \in L^{n/2}(\Omega,\R)$. 
The second Theorem extends the result in \cite{PS} to the case of incomplete spectral data, 
when $q \in L^p(\Omega,\R)$ and $p > n/2$, if $n=3$ and $p=n/2$, if $n \geq 4$. 
We do not consider the two dimensional case here. 
The character of the two dimensional problem is somewhat different, since
in this case $q \in L^1(\Omega,\R)$.

Theorem \ref{thm1} is proved by reducing it to the corresponding
inverse boundary value problem, which has been solved for $q \in L^{n/2}(\Omega)$
(see \cite{LN}, \cite{Cha} and \cite{DKS}).
The proof here is roughly of the same form as the argument in \cite{C}. 
Here we however use the $L^p$-theory of elliptic equations in 
a fairly systematic way, which enables us to handle unbounded potentials.

The proof of Theorem \ref{thm2} is based on the argument used in \cite{I}. 
Here one needs to consider a spectral parameter that goes to infinity in a specific way.
One needs moreover some form of $L^p$ resolvent estimates with an explicit dependence on 
the spectral parameter. The most interesting case here is $n=3$, since $n/2=3/2<2$. 
It turns out that for $q \in L^p(\Omega)$, with $p>n/2$, one can still use
$L^2$-theory and interpolation, to prove Theorem \ref{thm2}.
The case $p=3/2$ seems however to require better estimates,
where the spectral parameter $\lambda \in \C$ is allowed,
to grow more freely (than in e.g. Proposition \ref{agmon_result}), 
similar to the so called uniform $L^p$-estimates found in \cite{KRS} and \cite{DKS},
    
\medskip
\noindent
The paper is structured as follows. In section \ref{sec_apriori} we prove that the Dirichlet problem
for $-\Delta + q$ admits strong solutions, when considering appropriate boundary data. 
We also  derive some a-priori $L^p$-estimates. 
In section \ref{sec_s2DN} we prove Theorem \ref{thm1}. In section \ref{sec_incomplete}
we consider the case of incomplete spectral data and prove Theorem \ref{thm2}.
At the end of the paper we have included two 
appendices. The first one in section \ref{sec:spec} reviews some facts from spectral theory that we use.
The second appendix in section \ref{sec:app2} presents some basic facts concerning Besov spaces.

\section{ A priori estimates and strong solutions} \label{sec_apriori} 
\noindent
The aim of this section is to show in some detail 
that the Dirichlet problem for $-\Delta +q - \lambda$, $q \in L^{n/2}(\Omega,\R)$
admits strong solutions when considering boundary conditions from appropriate spaces. 
Recall that a strong solution is a solution in $W^{2,p}(\Omega)$, which satisfies the
equation almost everywhere.
Having strong solutions  will guarantee that 
we can later define the so called Dirichlet-to-Neumann map in a suitable way.
We end the section by deriving
an a priori estimate that has an explicit dependence on the 
spectral parameter $\lambda$.% which will of be importance in subsequent sections.

\medskip
\noindent
Our main aim will be to prove existence and uniqueness of solutions to the
Dirichlet problem
\begin{equation}
\begin{aligned}
\label{eq_bvp21}
( -\Delta +q - \lambda )u&=0\quad\textrm{in}\quad \Omega, \\
\gamma u&=f \quad\textrm{on}\quad \p\Omega,
\end{aligned}
\end{equation}
when $q \in L^{n/2}(\Omega,\R)$ and $f \in B^{2-1/p}_{pp}(\p \Omega)$.

\medskip
\noindent
We begin by deriving solutions to the corresponding inhomogeneous problem with
zero Dirichlet condition. To this end we will need a priori estimates in
$L^p$-norms, which we now state.

\begin{prop} \label{ADNthm}
Assume $q \in L^{n/2}(\Omega)$.
Suppose $u \in W^{2,p}(\Omega) \cap W^{1,p}_0(\Omega)$, $p = \frac{2n}{n+2}$, then
\begin{align} \label{eq_ADNest}
    \| u \|_{W^{2,p}(\Omega)} \leq C (\|( -\Delta +q) u \|_{L^p(\Omega)} + \| u \|_{L^p(\Omega)} ),
\end{align}
\end{prop}

\begin{proof}
By the estimate \eqref{eq:agmon_est}  of Lemma \ref{agmon_result} 
at the end of this section, we have that
\begin{align*}
    \| u \|_{W^{2,p}(\Omega)} &\leq C_0 \|( -\Delta + q - \lambda_0) u \|_{L^p(\Omega)} \\
                              &\leq C ( \|( -\Delta + q) u\|_{L^p(\Omega)} + \|u \|_{L^p(\Omega)})
\end{align*}
for some $\lambda_0 \in \R_{-}$.

\end{proof}

\noindent
The previous a priori estimate and the resolvent estimate \eqref{sobolevResEst} 
for Sobolev spaces can
be used to prove the following existence and uniqueness result
or the inhomogeneous problem, when $q \in L^{n/2}(\Omega)$ (see also Lemma 9.17 and Theorem
9.15 in \cite{GT}). Notice also that the Proposition applies to complex $\lambda$.

\begin{prop} \label{strong_prop}
Let $q \in L^{n/2}(\Omega)$ and $p=2n/(n+2)$. There exists a $\lambda_0 \in \R$
such that for $\lambda \in \C \setminus (\lambda_0,\infty)$,
there is an unique strong solution  $w \in W^{2,p}(\Omega) \cap W^{1,p}_0(\Omega)$ to
\begin{equation}
\begin{aligned}
\label{eq:bvp3}
(-\Delta + q - \lambda ) w &= F \quad\textrm{in}\quad \Omega, \\
\gamma w &= 0 \quad\textrm{on}\quad \p\Omega,
\end{aligned}
\end{equation}
for all $F \in L^p(\Omega)$. Moreover  we have that 
\begin{align} \label{eq:strong_est}
    \|w\|_{W^{2,p}(\Omega)} \leq   C_\lambda \|F\|_{L^p(\Omega)} .
\end{align}
\end{prop}
\begin{proof}
Pick  $F_k \in L^2(\Omega)$, s.t. $F_k \to F$, in the $L^p(\Omega)$-norm.
Denote by $u_k$ the corresponding solutions to the problem
\begin{equation*}
\begin{aligned}
(-\Delta + q - \lambda ) u_k &= F_k \quad\textrm{in}\quad \Omega, \\
\gamma u_k &= 0 \quad\textrm{on}\quad \p\Omega,
\end{aligned}
\end{equation*}
From the $L^2$-theory of elliptic partial differential operators, we know that  
that there is an $\lambda_0$ such that the solution $u_k$ exists and is unique,
when $\lambda \in \C \setminus (\lambda_0,\infty)$. 
By \eqref{eq_ADNest}  we have the estimate
\begin{align*} 
\| u_k \|_{W^{2,p}(\Omega)} \leq C_\lambda (\|( -\Delta +q -\lambda) u_k \|_{L^p(\Omega)} + \| u_k \|_{L^p(\Omega)} ).
\end{align*}
The solution is given by the resolvent, i.e. $u_k = R_q(\lambda)F_k \in H^1_0(\Omega)$. 
For the resolvent we have estimate \eqref{sobolevResEst}  
and by Sobolev embedding we have that
\begin{align*} 
\| R_q(\lambda) F_k \|_{L^\frac{2n}{n-2}(\Omega)} \leq C_\lambda \| F_k \|_{L^\frac{2n}{n+2}(\Omega)}.
\end{align*}
Combining the two previous estimates we get that 
\begin{align} \label{eq_weak}
    \| u_k \|_{W^{2,p}(\Omega)} & \leq C_\lambda (\| F_k \|_{L^p(\Omega)} 
    + \| R_q(\lambda) F_k \|_{L^p(\Omega)} ) \nonumber \\
    &\leq C_\lambda \| F_k \|_{L^p(\Omega)}   < C_\lambda < \infty % \\
\end{align}
Now $W^{2,p}(\Omega)\cap W^{1,p}_0(\Omega) \subset W^{2,p}(\Omega)$ is a complete and reflexive subspace, and
$\{u_k\}$ is a bounded set in this subspace because of \eqref{eq_weak}.
As a consequence we obtain a subsequence $w_k$ that converges weakly, i.e. 
$w_k\rightharpoonup w \in W^{2,p}(\Omega) \cap W^{1,p}_0(\Omega)$.
The weak convergence implies that 
\begin{align*} 
    \int_\Omega \varphi D^\alpha w_k \to 
    \int_\Omega \varphi D^\alpha w,
\end{align*}
for multi-indices $\alpha$, $|\alpha| \leq 2$ and $\varphi \in L^{2n/(n-2)}(\Omega)$.
This implies that
\begin{align}  \label{eq_weak_s1} 
    \int_\Omega \nabla \varphi \cdot \nabla (w_k - w) \varphi \to 0,
\end{align}
when $\varphi \in C^\infty_0(\Omega)$.
The Rellich-Kondrachov Theorem implies on the other hand that the embedding 
$id \colon W^{2,p}(\Omega) \to L^q(\Omega)$ is compact when $q < 2n/(n-2)$, so that $w_k \to w$ in the
$L^q(\Omega)$-norm, for $q< 2n/(n-2)$.
It follows in particular  that $w_k \to w$ in the $L^{n/(n-2)}(\Omega)$-norm. This and the 
Hölder inequality gives that
\begin{align} \label{eq_weak_s2} 
    \int_\Omega (q-\lambda)(w_k - w) \varphi 
    \leq \|(q-\lambda)\varphi \|_{L^\frac{n}{2}(\Omega)} \|w_k - w\|_{L^\frac{n}{n-2}(\Omega)} \to 0,
\end{align}
for $\varphi \in C_0^{\infty}(\Omega)$. 
It is straight forward to see using \eqref{eq_weak_s1}, \eqref{eq_weak_s2} and that
$F_k \to F$, in the $L^p(\Omega)$-norm, that $w$ is a weak solution to \eqref{eq:bvp3}. 
And thus $w$ a strong solution, since $w\in W^{2,p}(\Omega)$.

The weak solution $w$ is in $H^1_0(\Omega)$ and is a weak solution to
\eqref{eq:bvp3}, when $F \in L^p(\Omega)$ is
taken as an element in $H^{-1}(\Omega)$. The $L^2$-theory of elliptic operators implies 
that $w$ is the unique solution in $H^1_0(\Omega)$.
It follows that $w$ is the  unique solution in $W^{2,p}(\Omega) \cap W^{1,p}_0(\Omega)$.

\medskip
\noindent
It remains to prove the norm estimate. Assume
that estimate \eqref{eq:strong_est} is false, then there exists a sequences of 
functions  $F_k \in L^p(\Omega)$ and corresponding solutions  $u_k \in W^{2,p}(\Omega)$ s.t.
\begin{align*} 
\| u_k \|_{L^p(\Omega)} \geq k \| F_k \|_{L^p(\Omega)}.
\end{align*}
We may assume that  $\|u_k\|_{L^p(\Omega)}= 1$, so that $F_k \to 0$, in the $L^p(\Omega)$-norm.
Using \eqref{eq_ADNest} again, we have the estimate
\begin{align*} 
\| u_k \|_{W^{2,p}(\Omega)} \leq C_\lambda (\|( F_k \|_{L^p(\Omega)} + \| u_k \|_{L^p(\Omega)} ) < M < \infty.
\end{align*}
The set $\{u_k\}$ is hence bounded in the subspace  $W^{2,p}(\Omega) \cap W^{1,p}_0(\Omega)$. 
As earlier in the proof there is a subsequence $w_k$ that converges weakly, i.e. 
$w_k \rightharpoonup w_0 \in W^{2,p}(\Omega)$
in the $W^{2,p}(\Omega)$-norm. 
The Rellich-Kondrachov Theorem implies again that $w_k \to w_0$ in the
$L^q(\Omega)$-norm, for $q< 2n/(n-2)$. Which implies that $\|w_0\|_{L^p(\Omega)} = 1$,
since we picked $\|u_k\|_{L^p(\Omega)}= 1$.
Arguing as in the first part of the proof, we see that $w_0$ solves
\begin{equation*}
\begin{aligned}
(-\Delta + q - \lambda ) w_0 &= 0 \quad\textrm{in}\quad \Omega, \\
\gamma w_0 &= 0 \quad\textrm{on}\quad \p\Omega.
\end{aligned}
\end{equation*}
The solution to \eqref{eq:bvp3} is however unique and thus $w_0 \equiv 0$, which is a contradiction,
since  $\|w_0\|_{L^p(\Omega)} = 1$.

\end{proof}

\noindent
Recall that we can reduce the problem in \eqref{eq_bvp21} to the inhomogeneous problem
in \eqref{eq:bvp3}. 
The proof of the following Lemma is standard. We give it here as a convenience and since the argument
has a dependence on $\lambda$. 

\begin{lem} \label{red_lem}
The boundary value problem  \eqref{eq_bvp21} has a unique solution $u \in W^{2,p}(\Omega)$, 
$p=2n/(n+2)$ if the inhomogeneous problem
\begin{equation}
\begin{aligned}
\label{eq_bvp22}
(-\Delta + q - \lambda ) w &= F \quad\textrm{in}\quad \Omega, \\
\gamma w &= 0 \quad\textrm{on}\quad \p\Omega,
\end{aligned}
\end{equation}
admits a unique solution $w \in W^{2,p}(\Omega)$, for all
$F \in L^p(\Omega)$. 
If moreover the estimate 
\begin{align} \label{eq:red_est1}
    \|w\|_{W^{2,p}(\Omega)} \leq C \| F \|_{L^p(\Omega)}, 
\end{align}
holds, then we have the estimate   
\begin{align} \label{eq:red_est2}
    \|u\|_{W^{2,p}(\Omega)} \leq C_\lambda \| f \|_{B^{2-1/p}_{pp}(\p \Omega)}.
\end{align}
\end{lem}
\begin{proof}
Choose $F$ so that $F = (-\Delta+q - \lambda) E f$, 
where $E$ is the extension operator defined in section \ref{sec:app2}.
Notice that $F \in L^p(\Omega)$, because 
\begin{align*}
     \|\Delta E f\|_{L^p(\Omega)} \leq 
     C \| E f\|_{W^{2,p}(\Omega)},
\end{align*}
%because $E f \in W^{2m,p}(\Omega)$. 
and since by Sobolev embedding $W^{2,p}(\Omega) \subset L^\frac{2n}{n-2}(\Omega)$
and the H\"older inequality we have that
\begin{align*}
    \| q E f\|_{L^\frac{2n}{n+2}(\Omega)}   &\leq 
    \| q \|_{L^\frac{n}{2}(\Omega)}     
    \| E f\|_{L^\frac{2n}{n-2}(\Omega)}   \\
    &\leq
    \| q \|_{L^\frac{n}{2}(\Omega)}     
    \| E f \|_{W^{2,p}(\Omega)}.
\end{align*}
Let $w \in W^{2,p}(\Omega)$ be the unique solution 
of \eqref{eq_bvp22}, corresponding to $-F$. 
The function $u:= w + E f\in W^{2,p}(\Omega)$  solves then \eqref{eq_bvp21}. This
proves existence. 

Suppose on the other hand that $u_1$ and $u_2$ solve \eqref{eq_bvp21}. Then $u_1-u_2$ will be 
a solution to \eqref{eq_bvp22} with a zero source term. Uniqueness for \eqref{eq_bvp22} 
    now implies that $u_1=u_2$.

To  obtain  the norm estimate \eqref{eq:red_est2} we  use \eqref{eq:red_est1} and argue as follows.
\begin{align*} 
    \|u\|_{W^{2,p}(\Omega)}  % &= \|w + Ef\|_{W^{2,p}(\Omega)} \\
    &\leq  \| w \|_{W^{2,p}(\Omega)} + \| Ef\|_{W^{2,p}(\Omega)} \\
    &\leq  C \|(-\Delta+q - \lambda) E f \|_{L^p(\Omega)} + \| Ef\|_{W^{2,p}(\Omega)} \\
    &\leq  C_\lambda \| Ef\|_{W^{2,p}(\Omega)} \\
    &\leq  C_\lambda \| f \|_{B_{pp}^{2-1/p}(\p\Omega)}
\end{align*}

\end{proof}

\noindent
As a Corollary to Lemma \ref{red_lem} and Proposition \ref{strong_prop} we have the following 
existence result.

\begin{cor}\label{exi_uniq}
Let $q \in L^{n/2}(\Omega)$ and $p=2n/(n+2)$. There exists a $\lambda_0 \in \R$,
such that the Dirichlet problem \eqref{eq_bvp21} has a unique strong solution 
$u \in W^{2,p}(\Omega)$, when $\lambda \in \C \setminus (\lambda_0,\infty)$.
Moreover 
\begin{align} \label{eq:strong2_est}
    \|u\|_{W^{2,p}(\Omega)} \leq   C_\lambda \|f\|_{B_{pp}^{2-1/p}(\p\Omega)} .
\end{align}
\end{cor}

\medskip
\noindent
We end this section by formulating an $L^p$-apriori estimate
which has an explicit dependence on $\lambda$, when $\lambda \in \R$. 
This type of estimate was derived by Agmon in \cite{A1} (see Theorem 2.1.).
We need to modify it by adding  a $q\in L^{n/2}(\Omega)$ to the operator.
This estimate  is one of our main tools.

\begin{prop}\label{agmon_result}
Assume $\frac{2n}{n+2} \leq p < \infty$ and that 
$u \in W^{2,p}(\Omega) \cap W^{1,p}_0(\Omega)$. Then we have that
\begin{align} \label{eq:agmon_est}
    \sum_{j=0}^{2} |\lambda|^{\frac{2-j}{2}} \big \| u \big \|_{W^{j,p}{(\Omega)}}
    \leq C \big \| (-\Delta + q -\lambda)  u \big \|_{L^p{(\Omega)}},
\end{align}
when $-\lambda \in \R$ is large. The constant $C$ does not depend on $\lambda$.
\end{prop} 
\begin{proof} 
We begin by choosing a  $\mu \in \R_+$ s.t. $\mu^2 = -\lambda$.
Let $D= \Omega \times (-1,1)$ and let $\zeta \in C^\infty_0(-1,1)$ be s.t.
$\zeta(t) = 1$ when $|t| \leq 1/2$. Define $v(x,t):= \zeta(t) e^{i \mu t} u(x)$
and $\tilde q(x,t) := q(x)$.

By Theorem 2.1 and its proof in \cite{A1} (see also \cite{Ta}) there is a $\lambda_0$ s.t.
\[
    \|( -\Delta_x - \p_t^2 - \lambda_0) v \|_{L^p(D)}
    \geq C_0 \| v \|_{W^{2,p}(D)}.
\]
It follows that
\begin{align*}
    \|( -\Delta_x - \p_t^2 + \tilde q) v \|_{L^p(D)} &\geq  
    \|(-\Delta_x - \p_t^2- \lambda_0) v \|_{L^p(D)} - \|\tilde q v\|_{L^p(D)}\\
    &\quad-|\lambda_0| \| v \|_{L^p(D)}  \\
 &\geq  
C_0\| v \|_{W^{2,p}(D)} - \|\tilde q v\|_{L^p(D)} - |\lambda_0| \| v \|_{L^p(D)}. 
\end{align*}
Let $\tilde q_k \in C_0^\infty(D)$ be s.t. $\tilde q_k \to \tilde q$ in the $L^{n/2}(D)$-norm.
Then by the Hölder inequality and Sobolev embedding we have that 
\begin{align*}
\|\tilde q v\|_{L^p(D)} 
&\leq 
\|\tilde q- \tilde q_k \|_{L^{n/2}(D)} \| v \|_{L^{\frac{2n}{n-2}}(D)}
+\|\tilde q_k\|_{L^{\infty}(D)} \| v \|_{L^p(D)} \\
&\leq 
\frac{C_0}{2} \| v \|_{W^{2,p}(D)} + C_k \| v \|_{L^p(D)},
\end{align*}
when $k$ is choosen to be large enough. Combining the two previous estimates gives that
\begin{align*}
    \| ( -\Delta_x - \p_t^2 + \tilde q) v \|_{L^p(D)} &\geq  
    \frac{C_0}{2}\| v \|_{W^{2,p}(D)} - (C_k + |\lambda_0|) \| v \|_{L^p(D)}.
\end{align*}
Or in other words that
\begin{align} \label{eq_apriori}
    \| v \|_{W^{2,p}(D)} \leq C (\| (-\Delta_x-\p^2_t + \, \tilde q ) v \|_{L^p(D)} + \| v \|_{L^p(D)}),
\end{align}
for $p \geq 2n/(n+2)$
Using the definition of $v$ we have that
\[
    (-\Delta_x-\p^2_t) v = 
    \big(-\zeta \Delta_x u - (\zeta'' - \zeta \mu^2  + 2i\mu \zeta')u \big) e^{i \mu t},
\]
We can thus estimate  the first term on the right hand side of \eqref{eq_apriori} as
\begin{small}
\begin{align*}
     \| (-\Delta_x-\p^2_t + \,\tilde q  ) v \|_{L^p(D)} 
     &\leq C\big( \| (-\Delta_x  + q + \mu^2 ) u \|_{L^p(\Omega)} 
      + (1+|\mu|) \|u \|_{L^p(\Omega)} \big).
\end{align*}
\end{small}
Next consider the region $D' = \Omega \times (-\frac{1}{2}, \frac{1}{2})$. Since $\zeta = 1$ in $D'$,
we can write the Sobolev norm of $v$ as
\begin{align*}%\label{eq_apriori}
 \| v \|_{W^{2,p}(D')}
    & = (1+|\mu|+ |\mu|^2)\| u  \|_{L^p(\Omega)} \\ & \quad+ \sum_{j=1}^n (1+|\mu|)\| \p_j u  \|_{L^p(\Omega)}  
    + \sum_{j\leq i}^n \| \p_i \p_j u  \|_{L^p(\Omega)}. 
\end{align*}
Estimating the left hand side of \eqref{eq_apriori} from below, 
by taking the $L^p(D')$-norm and temporarily dropping the 
$L^p$-norms of the derivatives in the above expression, gives
\begin{align*}%\label{eq_apriori}
    |\mu|^2\| u \|_{L^p(\Omega)}  \leq
     C \big( \| (-\Delta_x  + q + \mu^2) u \|_{L^p(\Omega)} 
     + (2+|\mu|) \|u \|_{L^p(\Omega)} \big).
\end{align*}
We can absorb the second term on the right hand side by the first term by picking
a large $|\mu|$. We thus get that
\begin{align*}%\label{eq_apriori}
    |\mu|^2\| u  \|_{L^p(\Omega)} \leq C \| (-\Delta_x + q - \lambda ) u\|_{L^p(\Omega)}.
\end{align*}
By adding back the $\| \p_j u  \|_{L^p(\Omega)}$ and the
$\| \p_k\p_j u  \|_{L^p(\Omega)}$, that we dropped from the left hand side, yields
\begin{small}
\begin{align*}%\label{eq_apriori}
    |\mu|^2\| u  \|_{L^p(\Omega)} +
    |\mu|\| \p_j u  \|_{L^p(\Omega)} +  \| \p_i \p_j u  \|_{L^p(\Omega)} 
    \leq C \| (-\Delta_x + q - \lambda) u\|_{L^p(\Omega)},
\end{align*}
\end{small}
which implies the estimate of the claim.

\end{proof}

\noindent
Note that the above argument can be extended so that it applies to $\lambda$
that lie on certain rays emanating from the origin in $\C$ (see \cite{A1}).

\section{From spectral data to the Dirichlet-to-Neumann maps} \label{sec_s2DN} 

In this section we show that the spectral data determines the Dirichlet-to-Neumann map
$\Lambda_q(\lambda)$, when $-\lambda \in \R$ is large. 
This in conjunction with known results on the inverse boundary problem, 
will provide a proof for Theorem \ref{thm1}, which we spell out
at the end of the section.
\\

We begin by defining the Dirichlet-to-Neumann map.
Let $q \in L^{\frac{n}{2}}(\Omega,\R)$, $n\geq3$. 
Consider the Dirichlet problem
\begin{equation}
\begin{aligned}
\label{eq_bvp31}
(-\Delta+q - \lambda )u&=0, \quad\textrm{in}\quad \Omega, \\
u|_{\p\Omega}&=f, \quad\textrm{on}\quad \p\Omega,
\end{aligned}
\end{equation}
where
\begin{align*} 
f \in  B_{pp}^{2-1/p}(\p \Omega),
\end{align*}
with $p=2n/(n+2)$. % (see the appendix in section \ref{sec:app2}). 
The boundary value problem \eqref{eq_bvp31} has a unique solution 
$u \in W^{2,p}(\Omega)$ due to Lemma \ref{exi_uniq},
when $-\lambda$ is large.
The solution $u$ is furthermore bounded by the Dirichlet data, i.e.
\begin{align} \label{eq:intro_est}
    \|u\|_{W^{2,p}(\Omega)} \leq   C \|f\|_{B_{pp}^{2-1/p}(\p \Omega)}.
\end{align}
Recall that $\tilde \gamma$ is the Neumann trace operator, which gives meaning to 
the restriction $\nu \cdot \nabla u |_{\p \Omega}$ when $u$ is non-smooth.
We define the map $\Lambda_q$ as
\[
    \Lambda_q(\lambda) f := \tilde{\gamma}u, 
\]
where $u$ is the unique solution to \eqref{eq_bvp31} with boundary data $f$. The map
$\Lambda_q$ is called the Dirichlet-to-Neumann map or the DN-map for short.
Estimate 
\eqref{eq:intro_est} and the continuity of the Neumann trace operator,
shows that
$\Lambda_q(\lambda):  B_{pp}^{2-1/p}(\p \Omega)\to B_{pp}^{1-1/p}(\p \Omega)$ is continuous.

\medskip
\noindent
One of the main concerns of this section is to investigate 
the difference $[\Lambda_{q_1}(\lambda)-\Lambda_{q_2}(\lambda)]f$, as $\lambda \to -\infty$.
In considering the DN-maps we will need an estimate with an explicit dependence on $\lambda$
for the homogeneous problem \eqref{eq_bvp31}.
To this end it will be convenient to state the following Lemma, which is a direct consequence of 
Proposition \ref{agmon_result}.

\begin{lem}\label{c_indep_1} Let $w \in W^{2,p}(\Omega)$, $p=2n/(n+2)$ be a solution to \eqref{eq_bvp22}.
Then there is a constant $C$ independent of $\lambda$, for large $-\lambda \in \R$, such that
\begin{align} 
    &\| w  \|_{L^{p}{(\Omega)}} \leq C |\lambda|^{-1} \|F\|_{L^p(\Omega)}, \label{eq:lambda_est1}\\
    &\| w  \|_{W^{2,p}{(\Omega)}} \leq C \|F\|_{L^p(\Omega)}. \label{eq:lambda_est2}
\end{align}
\end{lem}
The estimate of the next Lemma is similar to the estimate of Lemma \ref{exi_uniq}.
Here we need  the constant in the estimate to be independent of $\lambda$, which requires 
some additional effort.

\begin{lem}\label{c_indep_2} Suppose $u \in W^{2,p}(\Omega)$, $p=2n/(n+2)$
solves  \eqref{eq_bvp31}.
Then there exists a constant $C$ independent of $\lambda$, such that
\begin{align}\label{eq:c_indep_est}
    \|u\|_{ L^{\frac{2n}{n-2}} (\Omega)} \leq C \| f \|_{ B_{pp}^{2-1/p}(\p\Omega)}.
\end{align}
\end{lem}

\begin{proof} 
Decompose $u$ as $u=v_0 + v_1$ where
\begin{equation}
\begin{aligned}
\label{eq:bvp5}
(-\Delta- \lambda ) v_0 &= 0 \quad\textrm{in}\quad \Omega, \\
\gamma v_0 &= f \quad\textrm{on}\quad \p\Omega.
\end{aligned}
\end{equation}
And 
\begin{equation}
\begin{aligned}
\label{eq:bvp6}
(-\Delta + q - \lambda ) v_1 &= -q v_0 \quad\textrm{in}\quad \Omega, \\
\gamma v_1 &= 0 \quad\textrm{on}\quad \p\Omega.
\end{aligned}
\end{equation}
By estimate \eqref{eq:lambda_est2} in Lemma \ref{c_indep_1} and Sobolev embedding we have that
\begin{align*}
 \|v_1\|_{ L^{\frac{2n}{n-2}} (\Omega)} \leq C 
 \|v_1\|_{ W^{2,p} (\Omega)} 
 &\leq C \| q v_0 \|_{L^{\frac{2n}{n+2}}(\Omega)} \\
                                         &\leq C \| q\|_{\frac{n}{2}(\Omega)}  \| v_0 \|_{\frac{2n}{n-2}(\Omega)},
\end{align*}
where $C$, independent of $\lambda$, for large $-\lambda$. It is therefore enough to obtain 
a bound on the $L^p$-norm for $v_0$, with $p=2n/(n-2)$ with a constant that is independent of $\lambda$.
We do this by making a second splitting. We set $v_0 = w_0 + w_1$, where
\begin{equation}
\begin{aligned}
\label{eq:bvp8}
(-\Delta - \tilde{\lambda} ) w_0 &= 0 \quad\textrm{in}\quad \Omega, \\
\gamma w_0 &= f \quad\textrm{on}\quad \p\Omega.
\end{aligned}
\end{equation}
where $\tilde{\lambda}$ is chosen with $-\tilde{\lambda}$ so large that the problem 
has a unique solution. The function $w_1$ solves 
\begin{equation}
\begin{aligned}
\label{eq:bvp9}
(-\Delta - \lambda ) w_1 &= (\lambda - \tilde{\lambda})w_0 \quad\textrm{in}\quad \Omega, \\
\gamma w_1 &= 0 \quad\textrm{on}\quad \p\Omega.
\end{aligned}
\end{equation}
We begin by estimating $w_0$, with writing it by means of an inhomogeneous problem, i.e. we
set $w_0 = \tilde{w} + Ef$, where $\tilde{w}$ is the unique solution to
\begin{equation}
\begin{aligned}
\label{eq:bvp10}
(-\Delta - \tilde{\lambda} ) \tilde{w} &= (\Delta + \tilde{\lambda})Ef \quad\textrm{in}\quad \Omega, \\
\gamma \tilde{w} &= 0 \quad\textrm{on}\quad \p\Omega.
\end{aligned}
\end{equation}
For $\tilde{w}$ we have by Proposition \ref{strong_prop} and the continuity of right inverse of 
the trace operator, that
\begin{align*}
    \| \tilde{w}\|_{W^{2,p}(\Omega)} &\leq C \| (\Delta + \tilde{\lambda})Ef \|_{L^{p}(\Omega)} \\
    &\leq C \| (\Delta + \tilde{\lambda})Ef \|_{L^{p}(\Omega)}  \\
    &\leq C \| Ef \|_{W^{2,p}(\Omega)} + |\tilde{\lambda}| \big\| Ef \big\|_{L^{p}(\Omega)}  \\
    &\leq C \| f \|_{B_{pp}^{2-1/p}(\p\Omega)}.
\end{align*}
Notice that the constant $C$ is independent of $\lambda$.
The above gives with Sobolev embedding the estimate
\begin{align*}
    \| w_0 \|_{L^\frac{2n}{n-2}(\Omega)} \leq C 
             \| w_0 \|_{W^{2,p}(\Omega)} &\leq C 
 \| \tilde{w}\|_{W^{2,p}(\Omega)} + \| Ef \|_{L^{p}(\Omega)} \\
    &\leq C \| f \|_{B_{pp}^{2-1/p}(\p\Omega)}.
\end{align*}
Now we can apply \eqref{eq:agmon_est} with \eqref{eq:bvp9} in the case $p=2n/(n-2)$.
This gives
\begin{align*}
    \| w_1 \|_{L^\frac{2n}{n-2}(\Omega)} &\leq 
    \frac{C}{|\lambda|}\| (\lambda- \tilde{\lambda}) w_0 \|_{L^\frac{2n}{n-2}(\Omega)} \\
    &\leq C \| w_0 \|_{L^\frac{2n}{n-2}(\Omega)}  \\
    &\leq C \| f \|_{B_{pp}^{2-1/p}(\p\Omega)}.
\end{align*}
Thus because $v_0 = w_0 + w_1$, we have that
\begin{align*}
    \| v_0 \|_{L^\frac{2n}{n-2}(\Omega)}
    \leq C \| f \|_{B_{pp}^{2-1/p}(\p\Omega)},
\end{align*}
where the constant $C$ is independent of $\lambda$, for large $-\lambda$.

\end{proof}

Next we examine the difference of the Dirichlet-to-Neumann maps, as $\lambda \to -\infty$.

\begin{prop} \label{DNmap_lim}
Let $p=2n/(n+2)$ and $\varepsilon > 0$. Then
\begin{align} \label{eq:DNmap_lim}
    \|\Lambda_{q_1}(\lambda) - \Lambda_{q_2}(\lambda) \|_{op} \to 0,
%    \|_{ \mathcal{L} ( B_{pp}^{2-1/p-\varepsilon},\;B_{pp}^{1-1/p-\varepsilon} )} \to 0,
\end{align}
when $\lambda \to - \infty$ and $0<\varepsilon \ll 1$,
where $\| \cdot \|_{op}$ denotes the operator norm on the 
space 
$\mathcal{L} ( B_{pp}^{2-1/p}(\p\Omega),\;B_{pp}^{1-1/p-\varepsilon}(\p\Omega))$,
of bounded linear operators. 
\end{prop}

\begin{proof} 
Let in $u_j\in W^{2,p}(\Omega)$, $j=1,2$ solve
\begin{equation}
\begin{aligned}
\label{eq:bvp_uj}
(-\Delta+q_j - \lambda )u_j&=0\quad\textrm{in}\quad \Omega, \\
\gamma u_j&=f \quad\textrm{on}\quad \p\Omega,
\end{aligned}
\end{equation}
    with $f \in B_{pp}^{2-1/p}(\p \Omega)$. Define $u := u_1-u_2$. Then $u$ will
solve
\begin{equation}
\begin{aligned}
\label{eq:bvp_u}
(-\Delta + q_1 - \lambda )u&= (q_2-q_1)u_2  \quad\textrm{in}\quad \Omega, \\
\gamma u &= 0 \quad\textrm{on}\quad \p\Omega.
\end{aligned}
\end{equation}
Estimate \eqref{eq:lambda_est1} of Lemma \ref{c_indep_1}, the Hölder inequality and the estimate
\eqref{eq:c_indep_est} for non zero boundary conditions, gives us now that
\begin{align} \label{eq:intpol_1}
    \| u  \|_{L^p(\Omega)} 
    &\leq \frac{C}{|\lambda|} \| u_2 \|_{L^{\frac{2n}{n-2}}(\Omega)} \| q_2-q_1\|_{L^{\frac{n}{2}}(\Omega)} \\
    &\leq \frac{C}{|\lambda|} \| f  \|_{B_{pp}^{2-1/p}(\p\Omega)} \nonumber ,
\end{align}
where $C$ is independent of $\lambda$, when $-\lambda$ is large.

Estimate \eqref{eq:lambda_est2}, the Hölder inequality and estimate
\eqref{eq:c_indep_est} give us likewise that
\begin{align} \label{eq:intpol_2}
\| u  \|_{W^{2,p}(\Omega)} 
\leq C \| u_2 \|_{L^{\frac{2n}{n-2}}(\Omega)} \| q_2-q_1\|_{L^{\frac{n}{2}}(\Omega)}
    \leq C \| f  \|_{B_{pp}^{2-1/p}(\p\Omega)},
\end{align}
where in both estimates the constants $C$ are independent of $\lambda$, when $-\lambda$ is large.

We also need the basic interpolation property of Sobolev spaces 
according to which
\begin{align} \label{eq:intpol_3}
    \| u\|_{W^{s,p}(\Omega)} \leq C \|u\|_{L^p(\Omega)}^{1-\frac{s}{2}} \|u\|_{W^{2,p}(\Omega)}^\frac{s}{2},
\end{align}
for $s \in (0,2)$.

From the definition of $u$ it follows that
\begin{align} \label{eq:DNmap_diff}
    \| \Lambda_{q_1}(\lambda)f - \Lambda_{q_2}(\lambda)f \|_{B_{pp}^{1-1/p-\varepsilon}(\p\Omega) }
    = \| \tilde{\gamma} u \|_{B_{pp}^{1-1/p-\varepsilon}(\p\Omega) },
\end{align}
To estimate the right hand side of \eqref{eq:DNmap_diff}, we use the definition 
of the Neumann trace operator. 
By the continuity of the normal traces, property \eqref{eq:intpol_3} and using estimates  
\eqref{eq:intpol_1} and \eqref{eq:intpol_2} we have  that
\begin{align*}
    \big\| \nu \cdot \nabla u|_{\p \Omega}\big \|_{B_{pp}^{1-1/p-\varepsilon}(\p\Omega) } 
    &\leq C  \| u\|_{W^{2 - \varepsilon,p}(\Omega)} \\
    &\leq C \|u\|_{L^p(\Omega)}^{\frac{\varepsilon}{2}} \|u\|_{W^{2,p}(\Omega)}^{1-\frac{\varepsilon}{2}}  \\
    &\leq \frac{C}{|\lambda|^\frac{\varepsilon}{2}}\| f  \|_{B_{pp}^{2-1/p}(\p\Omega)}
    \to 0,
\end{align*}
when $\lambda \to -\infty$. This together with \eqref{eq:DNmap_diff} shows that 
that the claim holds.

\end{proof}

We will prove Theorem \ref{thm1}, following the ideas in \cite{NSU}, \cite{C}
and \cite{KP}. We can view $\Lambda(\lambda)f$ as a holomorphic function of $\lambda$,
when $\lambda \not \in \spec ( -\Delta +q)$.
The following Lemma will imply 
that $[\Lambda_{q_1}(\lambda) -\Lambda_{q_2}(\lambda)]f$ is a polynomial in
a half-plane $\Re \lambda \leq \lambda_0$. The proof follows the argument in \cite{KP}.

\begin{lem}\label{lem_taylor} 
    Suppose the assumptions of Theorem \ref{thm1} hold and $f \in B^{2-1/p}_{pp}(\p\Omega)$.
    For every $m \in \N$, $m>(n+4)/2$ and $\lambda$, s.t. $-\lambda$ is large,
    we have that
    \[ 
        \frac{d^m}{d\lambda^m} \big[\, \Lambda_{q_1}(\lambda)f-\Lambda_{q_2}(\lambda)f\,\big] = 0.  
    \]
\end{lem}

\begin{proof}
Let $\lambda \in \C$ s.t. $ \Re \lambda \leq \lambda_0$, where $\lambda_0$ is s.t. 
there is a unique solution $u_{q}$ to the problem
\begin{equation} 
\begin{aligned}
\label{eq_bvp32}
    (-\Delta + q - \lambda )  u_{q} &= 0, \\
    \gamma  u_{q} &= f,
\end{aligned}
\end{equation}
for all $\lambda $ with $ \Re \lambda \leq \lambda_0$.
The solution $u_{q}$ can in general be expressed by means of the resolvent, by picking a extension 
$Ef$, of $f$ to $\Omega$ and then setting 
\begin{align} \label{eq_res_rep}
    u_{q} = Ef - R_q(\lambda)(-\Delta + q -\lambda)Ef.
\end{align}
Fix a $\tilde \lambda \in \R$,  such that $\tilde \lambda < \lambda$.
Consider an  extension $F\in W^{2,p}(\Omega)$ of $f$ to $\Omega$, given by the solution to the problem
\begin{align*}
    (-\Delta - \tilde \lambda )  F &= 0, \\
    \gamma  F &= f.
\end{align*}
Using \eqref{eq_res_rep} one can  write $u_q$ as
\begin{align*}
    u_{q} &= F - R_q(\lambda)(-\Delta + q - \lambda ) F \\
      &= F - \sum_{k} \frac{1}{\lambda_k - \lambda} 
      \big\langle (q +\tilde \lambda - \lambda ) F, \vp_k \big\rangle \vp_k 
\end{align*}
Now writing $F$  as  the series $ \sum_k (\vp_k,F)\vp_k$, we get that
\begin{align*} 
    u_{q} 
      &=  - \sum_{k} \frac{1}{\lambda_k - \lambda} 
      \big\langle (q +\tilde \lambda - \lambda_k ) F, \vp_k\big\rangle \vp_k.
\end{align*}
Taking the derivative in $\lambda$, gives then
\begin{align}\label{eq_urep}
    \frac{d^m}{d\lambda^m} u_{q}(\lambda) 
    &=  - m! \sum_{k} \frac{1}{(\lambda_k - \lambda)^{m+1}} 
    \big \langle (q +\tilde \lambda - \lambda_k ) F, \vp_k\big\rangle \vp_k
\end{align}
The sum in \eqref{eq_urep} converges in $L^2(\Omega)$ for every $m \in \N$. We need to show that it also
converges in $W^{2,p}(\Omega)$, for large $m$. 

Firstly by the Weyl law of Proposition \ref{prop_weyl} and estimate \eqref{eq_eigest} we know that 
\[
    \lambda_k \sim k^{2/n} \quad\text{ and }\quad \|\vp_k\|_{W^{2,p}(\Omega)} \leq C \lambda_k, 
\]
when $k$ is large.
Notice that $|\lambda_k - \lambda| \gtrsim |\lambda_k|$.
We can estimate the Sobolev norm of the individual  terms in 
the \eqref{eq_urep}, for large $k$ using these observations in the following manner
\begin{align*}
    \Big \| &\frac{1}{(\lambda_k - \lambda)^{m+1}} 
    \big \langle (q +\tilde \lambda - \lambda_k ) F, \vp_k\big\rangle \vp_k \Big \|_{W^{2,p}(\Omega)} \\
       &\leq C 
        \frac{1}{|\lambda_k|^{m+1}} 
       \|F\|_{L^\frac{2n}{n-2}(\Omega)}\|q+\tilde \lambda - \lambda_k\|_{L^{n/2}(\Omega)}
       \|\vp_k\|^2_{W^{2,p}(\Omega)} \\
       &\leq C
       k^\frac{-2(m+1)}{n} k^\frac{6}{n}.
\end{align*}
We need thus to choose $m$, so that $m > (n+4)/2$ in order to make the series in \eqref{eq_urep} converge in the
$W^{2,p}(\Omega)$-norm. 
It follows that 
\begin{align} \label{eq_Lambdarep}
    \tilde \gamma \frac{d^m}{d\lambda^m} u_{q}(\lambda) 
    &=  - m! \sum_{k} \frac{1}{(\lambda_k - \lambda)^{m+1}} 
    \big\langle (q +\tilde \lambda - \lambda_k ) F, \vp_k\big\rangle \tilde\gamma \vp_k ,
\end{align}
converges in the $L^2(\Omega)$-norm, when $m$ is chosen large enough.

The claim follows from \eqref{eq_Lambdarep}, since
$\tilde\gamma \vp_{q_1,k}= \tilde\gamma \vp_{q_2,k}$ and $\lambda_{q_1,k} = \lambda_{q_2,k}$, for 
every k, we have  that $\p_\lambda^m [\Lambda_{q_1} - \Lambda_{q_2}] (\lambda) f= 0$,
when $m$ is large enough.

\end{proof}

The above proof shows furthermore that the 
function $\lambda \mapsto \Lambda_{q}(\lambda)f$, is holomorphic in a half-plane
$\Re \lambda \leq \lambda_0 \in (-\infty,0)$, for some $\lambda_0$, since 
by \eqref{eq_Lambdarep}  we have a complex derivative  
$\p_\lambda^m \Lambda_{q} (\lambda) f= \nu \cdot \nabla \p_\lambda^m u_{q}(\lambda)|_{\p \Omega}$ exists,
when $m$ is  large enough, from which it follows that it exists for every $m \in \N$.

Lemma \ref{lem_taylor} implies that $[\Lambda_{q_1}(\lambda) - \Lambda_{q_2}(\lambda)]f$ 
is a polynomial in $\lambda$.
Lemma \ref{DNmap_lim} shows on the other hand that this polynomial goes to zero, as $\lambda \to -\infty$.
It follows that the polynomial in question is zero, so that $\Lambda_{q_1}(\lambda) = \Lambda_{q_2}(\lambda)$
for a fixed and large enough $-\lambda$.
It is however known that the Dirichlet-to-Neumann map determines uniquely a potential $q$ that is
in $L^{n/2}(\Omega)$ (see \cite{LN} and also \cite{Cha}, \cite{DKS}, \cite{KU}). It follows that $q_1 = q_2$,
which proves Theorem \ref{thm1}.

\section{Incomplete spectral data} \label{sec_incomplete}

\noindent
In this section we consider the case of incomplete spectral data and
prove Theorem \ref{thm2} by adapting the ideas in \cite{I}
to case when $q \in L^p(\Omega)$, with $p>n/2$, if $n=3$ and $p=n/2$, if $n\geq4$.

In the method used in \cite{I} one considers non-real values of the spectral parameter $\lambda$.
The arguments in the previous sections have dealt primarily  with real $\lambda$. Our first task 
is therefore to prove a variant of Proposition \ref{DNmap_lim} for certain complex values of $\lambda$.
For our purposes it will be enough to consider  $ \lambda$ 
in the set $\mathcal{D}_s\subset \C$, $s>0$ defined as 
\begin{align} \label{eq_defD}
    \mathcal{D}_s := \C \setminus
    \big (\{ \lambda \in \C \colon \Re \lambda \geq s\frac{1}{2}(\Im \lambda)^2 -1 \} \cup \spec (-\Delta+q) \big).
\end{align}

\begin{lem}\label{LambdaDlim} 
    Let $f \in B_{pp}^{2-1/p}(\p\Omega)$ and $q_1,q_2 \in L^{n/2}(\Omega)$.
    Then for $\lambda \in \mathcal{D}_s$, $s>0$, we have that
    \[ 
        \|\Lambda_{q_1}(\lambda)f-\Lambda_{q_2}(\lambda)f \|_{L^p(\p\Omega)} \to 0,
    \]
    as $|\lambda| \to \infty$.
\end{lem}

\begin{proof}
Let $\lambda \in \mathcal{D}_s$.
In the proof of Lemma \ref{lem_taylor} we derived the following expressions,
\begin{align}\label{eq_urep2}
    \frac{d^m}{d\lambda^m} u_{q}(\lambda) 
    &=  - m! \sum_{k} \frac{1}{(\lambda_k - \lambda)^{m+1}} 
    \big\langle (q +\tilde \lambda - \lambda_k ) F, \vp_k\big\rangle \vp_k,
\end{align}
where $u_q$ is a solution to \eqref{eq_bvp32}.
We can show that the sum converges in $W^{2,p}(\Omega)$, for large $m$ in the same
way we did in the proof of Lemma \ref{lem_taylor}.
This time we use that $|\lambda_k - \lambda| \gtrsim |\lambda_k|^{1/2}$, 
when $\lambda \in \mathcal{D}_s$, we get by
estimating as in the proof of Lemma \ref{lem_taylor}, that
\begin{align*}
    \Big \| &\frac{1}{(\lambda_k - \lambda)^{j+1}} 
    \big\langle (q +\tilde \lambda - \lambda_k ) F, \vp_k\big \rangle \vp_k \Big \|_{W^{2,p}(\Omega)} 
       \leq C
       k^\frac{-(m+1)}{n} k^\frac{6}{n}.
\end{align*}
We need thus to choose $m$ so that $m > n + 5$, in order to make the series in \eqref{eq_urep} converge in the
$W^{2,p}(\Omega)$-norm. 
It follows again that 
\begin{align} \label{eq_Lambdarep2}
    \tilde \gamma \frac{d^m}{d\lambda^m} u_{q}(\lambda) 
    &=  - m! \sum_{k} \frac{1}{(\lambda_k - \lambda)^{m+1}} 
    \big \langle (q +\tilde \lambda - \lambda_k ) F, \vp_k\big\rangle \tilde\gamma \vp_k ,
\end{align}
converges in the $L^2(\Omega)$-norm, when $m$ is chosen large enough.

We will rewrite equation \eqref{eq_Lambdarep2}, by integrating by parts as follows
\begin{align*}
    \big\langle (q +\tilde \lambda - \lambda_k ) F, \vp_k\big\rangle  \tilde \gamma \vp_k 
    = \int_{\p \Omega} \nabla_n \vp_k F \,dS \, \tilde \gamma \vp_k 
    =:  A_{q,k}.
\end{align*}
Equation \eqref{eq_urep2} gives that
\begin{align*}
    \frac{d^m}{d\lambda^m} \Lambda_q(\lambda)f 
    &=  - m! \sum_{k} \frac{A_{q,k}}{(\lambda_k - \lambda)^{m+1}} .
\end{align*}
when $m$ is large and $\lambda \in \mathcal{D}_s$. Because the spectral data 
is identical for the operators $-\Delta+q_j$, $j=1,2$, when $k \geq k_0$, we get that
\begin{align*}
    \frac{d^m}{d\lambda^m} \Big( \Lambda_{q_1}(\lambda)f -  \Lambda_{q_2}(\lambda)f \Big)
    &=   m! \sum_{k=1}^{k_0} \Big ( \frac{A_{q_2,k}}{(\lambda_{q_2,k} - \lambda)^{m+1}} 
         -\frac{A_{q_1,k}}{(\lambda_{q_1,k} - \lambda)^{m+1}} \Big). 
\end{align*}
By integrating  $m$-times in $\lambda$  we have
\begin{align*}
     \Lambda_{q_1}(\lambda)f -  \Lambda_{q_2}(\lambda)f 
    &=    \sum_{k=1}^{k_0} \Big( \frac{A_{q_2,k}}{\lambda_{q_2,k} - \lambda} 
     -\frac{A_{q_1,k}}{\lambda_{q_1,k} - \lambda} \Big) + \sum_{k=1}^{k_0}  \lambda^{m-1} C_{q_1,q_2,k},
\end{align*}
where $C_{q_1,q_2,k} \in B^{1-1/p}_{pp}(\p\Omega)$. The left hand side will go to zero in
the $L^p(\p\Omega)$, when considering the special case $\lambda \in \R$ and $\lambda \to -\infty$ 
because of Proposition \ref{DNmap_lim}. The same applies to the first
term on the right hand side. It follows that $\sum_k C_{q_1,q_2,k}=0$.
We hence see that for $\lambda \in \mathcal{D}_s$
\begin{align*}
    \|\Lambda_{q_1}(\lambda)f -  \Lambda_{q_2}(\lambda)f \|_{L^p(\p\Omega)} \to 0,
\end{align*}
as $|\lambda| \to \infty$.

\end{proof}

\medskip
\noindent
Following \cite{I} we will consider 
\[
    \varphi_{\lambda,\omega} := e^{i \sqrt{\lambda} \omega \cdot x},
\]
where $\lambda  \in \C \setminus (0,\infty)$ and $\omega \in S^{n-1}$. 
Moreover we define 
\[
    S(\lambda,\theta,\omega;q) := 
    \int_{\p \Omega} \Lambda_{q}(\lambda)\varphi_{\lambda,\omega} \varphi_{\lambda,-\theta}\,dS_x
\]
Let  $R_q(\lambda)$ be the resolvent operator related to the Dirichlet problem \eqref{eq_bvp21}.
The following Lemma was established in \cite{I} for $q \in L^\infty(\Omega)$ (see Lemma 2.2 in \cite{I}). The
proof is essentially the same for $q \in L^{n/2}(\Omega)$.

\begin{lem} \label{born_lem}
We have the following identity,
\begin{align*}
    S(\lambda,\theta,\omega;q) 
    =  
        &\int_\Omega  e^{-i \sqrt{\lambda} (\theta- \omega )\cdot x}q(x) \,dx 
        -\frac{\lambda}{2}(\theta-\omega)^2 \int_\Omega e^{-i \sqrt{\lambda} (\theta- \omega )\cdot x} \,dx \\
        &- \big\langle R_{q}( \lambda)(q\varphi_{\lambda,\omega}), 
        \, \ov{q\varphi_{\lambda,-\theta}} \big\rangle,
\end{align*}
where $\langle \cdot, \cdot \rangle$ is the duality pairing.
\end{lem}

\medskip
\noindent
Our aim is to use the above Lemma to obtain the Fourier transform of the difference of the  potentials.
The first step is to  choose the parameters $\lambda$, $\omega$ and $\theta$ in a suitable way.
More precisely we shall make the following choices in accordance with \cite{I}.

Let $0 \neq \xi \in \R^n$ be fixed and $\eta \in S^{n-1}$ and $\xi \cdot \eta = 0$.
We will consider a specific $\theta$ and $\omega$ depending on a parameter $m \in \N$.
These are  chosen as follows
\begin{align} \label{eq_isozakiparam}
    \begin{cases}
        \theta(m) &:=  C(m) \eta + \xi/2m, \\
        \omega(m) &:= C(m) \eta - \xi/2m, \\
        \sqrt{\tau(m)} &:= m + i , 
    \end{cases}
\end{align}
where $C(m) := \big(1 - \frac{|\xi|^2}{4m^2} \big)^{1/2}$, so that $\theta(m), \omega(m) \in S^{n-1}$. 
It follows that 
\begin{align*}
    \begin{cases}
        \sqrt{\tau(m)} \big(\theta(m) - \omega(m)\big) \to  \xi,  \nonumber \\
        \Im \tau(m) \to \infty,  \\
        \Im \sqrt{\tau(m)}\theta(m), \; \Im \sqrt{\tau(m)}\omega(m) \leq C < \infty, \nonumber
    \end{cases}
\end{align*}
as $m \to \infty$. We will furthermore use the abbreviations
\[
    \psi_\omega :=  \varphi_{\tau(m),\omega(m)}  \quad \text{and}  \quad
    \psi_\theta :=  \varphi_{\tau(m),-\theta(m)}.
\]
Lemma \ref{LambdaDlim} implies now the following.

\begin{lem} \label{bndry_decay}
We have that
\begin{align*}
        S(\tau,\theta,\omega;q_1) - S(\tau,\theta,\omega;q_2) \to 0,
\end{align*}
as  $m \to \infty$.
\end{lem}
\begin{proof} Using the definition of $S$ we see that we need to show that
\[
    \int_{\p \Omega} (\Lambda_{q_1} \big(\tau)\psi_\omega 
    - \Lambda_{q_2}(\tau)\psi_\omega \big) \psi_\theta\,dS_x
    \to 0,
\]
Where $\tau=\tau(m)$ and $s>0$, are such  that $\tau(m) \in \mathcal{D}_s$. In addition we have that
    $\|\psi_\theta\|_{L^\infty}=\|\varphi_{\tau(m),-\theta(m)}\|_{L^\infty} \leq C < \infty$, when $m\to \infty$. 
The claim follows now from Lemma \ref{LambdaDlim}, by the Hölder inequality.

\end{proof}

Lemmas \ref{born_lem} and \ref{bndry_decay} imply that
\begin{small}
\begin{align} \label{eq_intlim}
    \int_\Omega  e^{-i \sqrt{\tau} (\theta- \omega )\cdot x}(q_1-q_2) +
    \sum_{j=1,2} (-1)^{j}
    \langle q_j R_{q_j}(\tau)(q_j \psi_\omega), \, \ov{\psi_{\theta}} \big \rangle \to 0,
\end{align}
\end{small}
as $m \to \infty$, when $\theta,\tau$ and $\omega$ is chosen as in \eqref{eq_isozakiparam}.
If we can now show that the two terms containing the Dirichlet resolvents $R_{q_j}$ vanish in the
limit, we then obtain that
\begin{align*}
    \int_\Omega  e^{i \xi \cdot x}(q_1-q_2)\,dx = 0,
\end{align*}
and thus that $q_1 = q_2$, which proves Theorem \ref{thm2}. It remains therefore to analyze the 
terms in \eqref{eq_intlim} containing the resolvents. To this end we derive  the following resolvent 
estimate.

\medskip
\noindent
\begin{rem} \label{rem_pos}
We can assume  for simplicity that the operators $(-\Delta+q_j)$ are positive in the sense
that
\begin{align}\label{eq_pos}
    C  \| u \|_{H^1(\Omega)}^2 \leq \big\langle (-\Delta + q_j) u,u\big \rangle, 
\end{align}
when $u \in  H^1_0(\Omega)$,
since the above inequality holds for $(-\Delta+q_j+\lambda_0)$, 
where $\lambda_0$ is a suitable constant.
This can been seen by using \eqref{eq_coercive}. 
Adding the $\lambda_0$, only shifts the spectrum of the operators $(-\Delta + q_j)$,
so that  the spectral data for $(-\Delta+q_j+\lambda_0)$ coincide for $j=1,2$, provided
that the spectral data for $(-\Delta+q_j)$, for $j=1,2$ coincide.
\end{rem}

\begin{lem} \label{2ndRes}
Assume that $q\in L^{n/2}(\Omega,\R)$ is such that \eqref{eq_pos} holds, and that $\tau(m) = (m+i)^2$.
Then  for $f \in L^p(\Omega)$, $p=2n/(n+2)$, we have that
\begin{align} \label{eq_Rgrows}
    \| R_q(\tau(m)) f \|_{L^\frac{2n}{n-2}(\Omega)} 
    \leq C |\Im \tau(m)| \big \| f \big \|_{L^\frac{2n}{n+2}(\Omega)},
\end{align}
    where $C$ is independent of $\tau(m)$.

\end{lem}
\begin{proof} By \eqref{eq_pos} we have that
\begin{align} \label{eq_H1_bellow}
    C  \| u \|_{H^1(\Omega)}^2 \leq \big\langle (-\Delta + q)u,u\big\rangle
   = \Big\langle \sum_k \lambda_k \langle \varphi_k ,u \rangle \varphi_k , \,u \Big\rangle,
\end{align}
which holds for $u \in H^1_0(\Omega)$.
Let $f \in C^\infty(\Omega)$.
By \eqref{eq_sumrep} 
%Corollary 2.39 in \cite{Mc},
we know that the resolvent can be expressed as the sum
\begin{align} \label{eq_sumrep}
    R_q(\lambda) f &= \sum_{k=1}^\infty \frac{ \langle \varphi_k,f \rangle }{\lambda_k-\lambda} \varphi_k,
\end{align}
    that is convergent in the $L^2(\Omega)$-norm. 
Using this and \eqref{eq_H1_bellow} we get that
\begin{align} \label{eq_sum1}
    \| R_q(\lambda) f \|_{H^1(\Omega)}^2 
    &\leq 
    \sum_{j=1}^\infty 
    \lambda_j\Big| \Big(\varphi_j,
\sum_{k=1}^\infty \frac{\langle\varphi_k,f\rangle}{\lambda_k-\lambda} \varphi_k \Big)_{L^2}\Big|^2  \nonumber \\
    &\leq C
    \sum_{k=1}^\infty  \lambda_k \frac{|\langle\varphi_k,f\rangle|^2}{|\lambda_k-\lambda|^2}\\ 
    &\leq C \sup_{k} \Big| \frac{\lambda_k}{\lambda_k - \lambda} \Big|^2
    \sum_{k=1}^\infty \frac{|\langle \varphi_k,f \rangle|^2}{\lambda_k}. \nonumber  
\end{align}
Taking the Fourier representation and using \eqref{eq_sumrep} we get that
\begin{align*}
\big\langle f,R_q(0) f \big \rangle = 
\sum_{j,k} \Big( \langle\varphi_j, f\rangle \varphi_j, \frac{\langle\varphi_k,f\rangle}{\lambda_k} \varphi_k \Big)
= \sum_{k=1}^\infty \frac{|\langle\varphi_k,f\rangle|^2}{\lambda_k}.
\end{align*}
So that by the continuity of $R_q(0)$ we have that
\begin{align} \label{eq_sum2}
    \sum_{k=1}^\infty \frac{|\langle \varphi_k,f \rangle|^2}{\lambda_k} 
    &\leq C \|f\|_{H^{-1}(\Omega)} \|R_q(0)f\|_{H^{1}_0(\Omega)} \\ 
    &\leq C \|f\|_{H^{-1}(\Omega)}^2. \nonumber 
\end{align}
Combining \eqref{eq_sum1} and \eqref{eq_sum2}, gives then
\begin{align} \label{eq_supest}
    \| R_q(\lambda) f \|_{H^1(\Omega)}^2
    &\leq C \sup_{k} \Big| \frac{\lambda_k}{\lambda_k - \lambda} \Big|^2
    \|f\|_{H^{-1}(\Omega)}^2, 
\end{align}
for $f \in C^\infty(\Omega)$, where $C$ does not depend on $\lambda$. Using a density
argument shows that this holds also when $f \in H^{-1}(\Omega)$.

    Assume now that $\lambda =  \tau(m) = (m+i)^2$ and consider the expression
\begin{align*}
    \sup_{k} \Big| \frac{\lambda_k}{\lambda_k - \tau(m)} \Big|
    &\leq C
    \sup_{k} \frac{|\lambda_k|}{|\lambda_k - m^2 + 1| + |2mi|}.
\end{align*}
The function $f(x) = x /(|x - m^2 + 1| + |2m|)$
is monotonely decreasing for $x \geq m^2-1$ and 
$f(m^2-1) \leq m$. 
For $0 < x \leq m^2-1$ we have that $f(x) \leq (m^2-1)/m \lesssim m$. 
Since $\lambda_k > 0$ we have the estimate
\begin{align*}
    \sup_{k} \Big| \frac{\lambda_k}{\lambda_k - \tau(m) } \Big|
    &\leq C |\Im \tau(m)|.
\end{align*}
The claim follows now from \eqref{eq_supest}, by setting $\lambda = \tau(m)$ and using 
Sobolev embedding.

\end{proof}

\begin{rem}
One would expect that the above $L^2$-theory based estimate could be improved,
since it is well known that in the case of $\Omega= \R^n$, or when $\Omega$
is a Riemannian manifold without boundary, and $q=0$ one
has so called "uniform Sobolev estimates", see e.g. \cite{KRS} and \cite{DKS}.
Estimates like Proposition \ref{agmon_result} also seem suggest that there is room for improvement.
We are however not aware of any such uniform estimates for 
the Dirichlet resolvent for domains with a boundary which would include $\tau(m)$, when 
$m$ is large.
\end{rem}

\medskip
\noindent
We  are now ready to show that the resolvent terms in \eqref{eq_intlim} vanish.

\begin{lem}
    Assume that $q \in L^p(\Omega,\R)$, with $p=n/2$, if $n \geq4$ and $p>n/2$, if $n=3$,
    then we have that
\begin{align*}
    &\quad\big 
    \langle q R_{q}(\tau(m))(q \psi_\omega), \, \ov{\psi_\theta} \big \rangle 
    \to 0,
\end{align*}
as $m \to \infty$.
\end{lem}

\begin{proof}
It is enough to show that the $L^1$-norm of the first term in the duality pairing on the left 
hand side of the claim goes to zero, since $\vp_{\tau(m),-\theta(m)} < C < \infty$, where $C$
can be picked to be independent of $m$.

By the H\"older inequality, we have that
\begin{align} \label{eq_holder}
    \|q R_{q}(\tau(m))(q \psi_\omega)\|_{L^1(\Omega)}
    &\leq \|q\|_{L^p(\Omega)} \|R_{q}(\tau(m))(q \psi_\omega)\|_{L^{p^*}(\Omega)}. 
\end{align}
It will be convenient to write $p$, as $p=\frac{n+\eps}{2}$, for some $\eps>0$ and the Hölder conjugate $p^*$,
as $p^* = \frac{n+\eps}{n+\eps-2}$.

Suppose firstly that $n \geq 4$. In this case $p^*\leq  2 \leq p$.
Estimate  \eqref{basicResEst} gives us immediately that
\begin{align} \label{eq_case1}
    \|  R_q(\tau(m)) (q \psi_\omega)\|_{L^\frac{n+\eps}{n+\eps-2}(\Omega)} 
    \leq  \frac{C}{|\Im \tau(m)|}  \| q \psi_\omega \|_{L^{\frac{n+\eps}{2}}(\Omega)} \to 0,
\end{align}
as $m \to \infty$. Notice that this is true even when $\eps = 0$, proving the claim when $n \geq 4$.

Assume now that $n=3$.
In this case we use the Riesz-Thorin interpolation theorem to obtain an estimate for the $L^{p^*}$-norm
above. The Riesz-Thorin interpolation theorem  states that
\begin{align*}
\| T \varphi \|_{L^{q_\theta}} 
    \leq C_0^\theta C_1^{1-\theta} \| \varphi \|_{L^{p_\theta}},
\end{align*}
where 
$p_\theta^{-1} = \theta p_0^{-1}  + (1-\theta) p_1^{-1} $ and
$q_\theta^{-1}  = \theta q_0^{-1}  + (1-\theta) q_1^{-1} $,
provided we have the estimates 
$\| T \vp \|_{L^{q_j}} \leq C_j \| \varphi \|_{L^{p_j}}$, for $j=0,1$.
By Lemma \ref{2ndRes} and \eqref{basicResEst} we have that
\begin{align*} 
    \| R_q(\tau(m)) f \|_{L^2(\Omega)} &\leq \frac{C}{|\Im \tau(m)|} \| f \|_{L^2(\Omega)}, \\
    \| R_q(\tau(m)) f \|_{L^\frac{2n}{n-2}(\Omega)} &\leq C |\Im \tau(m)|  \| f \|_{L^\frac{2n}{n+2}(\Omega)}.
\end{align*}
Applying the Riesz-Thorin interpolation Theorem to these estimates and taking
$\theta \in (0,1)$ to be $\theta = \frac{3-3\eps}{6+2\eps}$, gives that  
\begin{align*}
    \|  R_q(\tau(m)) (q \psi_\omega)\|_{L^\frac{n+\eps}{n+\eps-2}(\Omega)} 
    \leq  C |\Im \tau(m)|^{2\theta-1} \| q \psi_\omega \|_{L^{\frac{n+\eps}{2}}(\Omega)} \to 0,
\end{align*}
as $m \to \infty$,
since $\theta < 1/2$. This together with \eqref{eq_case1} and \eqref{eq_holder} shows that  
\begin{align*}
    \langle q R_{q}(\tau(m))(q \psi_\omega), \,\ov{\psi_\theta} \big\rangle
    \to 0,
\end{align*}
as $m \to \infty$. 

\end{proof}

\section{Appendix A. The spectrum} \label{sec:spec}

\noindent
In this section we review some basic facts from the spectral theory relating to the operator
$-\Delta+q$, with $q \in L^{n/2}(\Omega,\R)$.
A weak solution $u \in H^1_0(\Omega)$ to \eqref{eq_bvp22} is a  function
for which
\begin{align*}
    \Phi(u,v) := \int_{\Omega} \nabla u \cdot \nabla \ov{v} + qu\ov{v} = \int_{\Omega} F \ov{v},
\end{align*}
holds for every $v \in H^1_0(\Omega)$.

The sesquilinear form $\Phi: H^1_0(\Omega)\times H^1_0(\Omega) \to \C$ is continuous, since
by the Hölder inequality and Sobolev embedding we have that
\begin{align*}
    |\Phi(u,v)| 
    &\leq 
    \|u\|_{H^1(\Omega)}\|v\|_{H^1(\Omega)} 
    + \|q\|_{L^{n/2}(\Omega)}\|u\|_{L^\frac{2n}{n-2}(\Omega)}\|v\|_{L^\frac{2n}{n-2}(\Omega)} \\
    &\leq C \|u\|_{H^1(\Omega)}\|v\|_{H^1(\Omega)}.
\end{align*}
The form $\Phi$ is in addition coercive, which can be seen as follows. Let $q_k\in C^\infty(\Omega)$,
be such that $q_k \to q$, in the $L^{n/2}(\Omega)$-norm. Then
\begin{align} \label{eq_coercive}
    \Phi(u,u)  &\geq 
    \|u\|_{H^1(\Omega)}^2
    - \|q_k\|_{L^\infty(\Omega)}\|u\|_{L^2(\Omega)}^2  
    - \|q-q_k\|_{L^{n/2}(\Omega)}\|u\|^2_{L^\frac{2n}{n-2}(\Omega)} \nonumber \\
    &\geq C \|u\|^2_{H^1(\Omega)} - C_0\|u\|_{L^2(\Omega)}^2.
\end{align}
The operator $L:=-\Delta + q:H^1_0(\Omega) \to H^{-1}(\Omega)$
can be understood as the operator given by $\langle L u, v \rangle =  \Phi(u,v)$,
where $\langle \cdot , \cdot \rangle$ denotes the duality pairing.
It follows that $L$ is also coercive and continuous.
The adjoint $L^*$ of $L$ may be defined as $\langle L^* u, v \rangle = \ov{\Phi(u,v)}$.
It follows then that $L$ is self-adjoint on $H^1_0(\Omega)$.

We have moreover that $H^1_0(\Omega) \subset L^2(\Omega) \subset H^{-1}(\Omega)$, where
$L^2(\Omega)$ is called the pivot space for $H^1_0(\Omega)$.
Since $L:H^1_0(\Omega) \to H^{-1}(\Omega)$ is bounded, coercive and self-adjoint,
and $L^2(\Omega)$ is a pivot space,
we have by Theorem 2.37 in \cite{Mc} firstly that
\begin{align*} 
    &\text{There is sequence of eigenfunctions $\vp_k \in H^1_0(\Omega)$ and corresponding } \\ &\text{
        eigenvalues 
        $-\infty < \lambda_1 \leq \lambda_2 \leq \dots  \leq \lambda_k \to \infty$.}
\end{align*}
As a second consequence of  Theorem 2.37 in \cite{Mc} is that
\begin{align*} 
    \text{The set $\{\vp_k\}$ is a complete orthonormal basis in $L^2(\Omega)$.}
\end{align*}
The Sobolev norm of the eigenfunctions also have nice estimates. 
By the estimate of Proposition \ref{ADNthm} we have that
\begin{align}\label{eq_eigest}
    \|\vp_k\|_{W^{2,p}(\Omega)} \leq C_0 (|\lambda_k|+1) \|\vp_k\|_{L^p(\Omega)}.
\end{align}
For $\lambda \in \C$, s.t.  $\lambda \notin \{\lambda_k\} = \spec(-\Delta + q)$ we have that the 
Dirichlet resolvent, i.e. the operator
$R_q(\lambda) := (-\Delta +q -\lambda)^{-1} \colon H^{-1}(\Omega) \to H^1_0(\Omega)$ is continuous. 
Hence we have the following estimate
\begin{align} \label{sobolevResEst}
    \| R_q(\lambda) f \|_{H^1(\Omega)} \leq C_\lambda \| f \|_{H^{-1}(\Omega)}.
\end{align}
The resolvent can be expressed as the sum
\begin{align} \label{eq_sumrep}
    R_q(\lambda) f &= \sum_{k=1}^\infty \frac{ \langle \varphi_k,f \rangle }{\lambda_k-\lambda} \varphi_k,
\end{align}
which is convergent in the $L^2(\Omega)$-norm 
(see for instance Corollary 2.39 in \cite{Mc}). From this one can furthermore 
derive the norm estimate
\begin{align} \label{basicResEst}
    \| R_q(\lambda) f \|_{L^2(\Omega)} \leq \frac{C}{|\Im \lambda|} \| f \|_{L^2(\Omega)}.
\end{align}

A further fact we need concerning the spectrum of the operator $-\Delta+q$  is the
following Weyl law, that pertains to potentials  $q \in L^{n/2}(\Omega,\R)$.

\begin{prop}\label{prop_weyl}
Let $q\in L^{n/2}(\Omega,\R)$, $n\geq3$. 
Then for the Schr\"odinger operator $-\Delta+q$, with form domain $H^1_0(\Omega)$,
we have the Weyl law
\[
    \lambda_k\sim\frac{4\pi^2}{(V_B\,V_\Omega)^{2/n}}\,k^{2/n}, % \qquad(k\longrightarrow\infty)
\]
as $k\to \infty$, 
where $V_B$ is the volume of the $n$-dimensional unit ball and $V_\Omega$ is
the volume of $\Omega$.

\end{prop}
\begin{proof}
    Denote the Dirichlet eigenvalues for the Laplacian (i.e. when $q=0$), by
$\lambda_1^D\leqslant\lambda_2^D\leqslant\ldots$. 
These exhibit the following Weyl law,
\[
    \lambda_k^D\sim\frac{4\pi^2}{(V_B\,V_\Omega)^{2/n}}\,k^{2/n},
\]
as $k\to \infty$
(see e.g.\ Sect.\ 13.4 of \cite{Zuily}).

Next, the mini-max principle (see e.g.\ Sect.\ 4.5 in \cite{Davies}) says that
\[
    \lambda_k=\min_{\substack{X\subseteq D\\\dim X=k}}\max_{\substack{\varphi\in X\\\left\|\varphi\right\|=1}}
    \int\limits_\Omega\left(\left|\nabla\varphi\right|^2+q\left|\varphi\right|^2\right)
\]
and
\[
    \lambda_k^D=\min_{\substack{X\subseteq D\\\dim X=k}}\max_{\substack{\varphi\in X\\\left\|\varphi\right\|=1}}
    \int\limits_\Omega\left|\nabla\varphi\right|^2,
\]
where the minima are to be taken over vector subspaces $X$ of $D$.  By Lemma \ref{lem_formbnd} 
there exists a constant
$C_\varepsilon\in\mathbb R_+$, for $\varepsilon > 0$, such that
\begin{align}\label{eq_formbnd}
    \int\limits_\Omega
    q\left|\varphi\right|^2\leqslant\varepsilon\bigl\|\nabla\varphi\bigr\|^2_{L^2(\Omega)}
    +C_\varepsilon\bigl\|\varphi\bigr\|^2_{L^2(\Omega)},
\end{align}
for $\varphi\in D$.
Combining this relative bound with the mini-max principle gives
\begin{align*}
\lambda_k&=\min_{\substack{X\subseteq D\\\dim X=k}}\max_{\substack{\varphi\in
    X\\\left\|\varphi\right\|=1}}\int\limits_\Omega\left(\left|\nabla\varphi\right|^2+q\left|\varphi\right|^2\right)\\
&\leqslant\min_{\substack{X\subseteq D\\\dim X=k}}\max_{\substack{\varphi\in
    X\\\left\|\varphi\right\|=1}}\int\limits_\Omega\left(\left(1+\varepsilon\right)\left|\nabla\varphi\right|^2+C_\varepsilon\left|\varphi\right|^2\right)
=\left(1+\varepsilon\right)\lambda_k^D+C_\varepsilon.
\end{align*}
Similarly, we obtain
\[\lambda_k\geqslant\left(1-\varepsilon\right)\lambda_k^D-C_\varepsilon.\]
Thus,
\[\left(1-2\varepsilon\right)\lambda_k^D\leqslant\lambda_k\leqslant\left(1+2\varepsilon\right)\lambda_k^D\]
for sufficiently large $k\in\mathbb Z_+$. Since $\varepsilon$ was arbitrarily
small, we have established that $\lambda_k\sim\lambda_k^D$.

\end{proof}

\noindent
We need to justify the use of \eqref{eq_formbnd} in the previous proof, in order to finish it.
This is done in the following Lemma.

\begin{lem}\label{lem_formbnd}
Let $q \in L^{n/2}(\Omega,\R)$, $n\geq3$ and $\vp \in H^1_0(\Omega)$. Then for 
$\varepsilon > 0$, we have 
\begin{align*}
    \int_\Omega q \vp^2 \leq \varepsilon\|\nabla \vp\|_{L^2(\Omega)}^2 + C_\varepsilon \| \vp \|_{L^2(\Omega)}^2.
\end{align*}
\end{lem}
\begin{proof}
By the Hölder inequality we have that
\begin{align*}
    \int_\Omega q \vp^2 \leq \|q\vp\|_{L^\frac{2n}{n+2}(\Omega)}\| \vp \|_{L^\frac{2n}{n-2}(\Omega)}.
\end{align*}
By the Poincaré inequality and Sobolev embedding we have that
\begin{align} \label{eq_split} 
    \int_\Omega q \vp^2 &\leq C\|q\vp\|_{L^\frac{2n}{n+2}(\Omega)}\| \nabla \vp \|_{L^2(\Omega)} \nonumber\\
    &\leq \frac{C}{\varepsilon}\|q\vp\|_{L^\frac{2n}{n+2}(\Omega)}^2 
    + \varepsilon\| \nabla \vp \|_{L^2(\Omega)}^2.
\end{align}
Next we pick a smooth approximation $q_k\in C^\infty(\Omega)$, s.t.
$q_k \to q$, in $L^{n/2}(\Omega)$. We have that
\begin{align*}
    \|q\vp\|_{L^\frac{2n}{n+2}(\Omega)}^2 
    &\leq 
    2\|q_k\vp\|_{L^\frac{2n}{n+2}(\Omega)}^2 
    + 2\|(q-q_k)\vp\|_{L^\frac{2n}{n+2}(\Omega)}^2  \\
    &\leq 
    2\|q_k\|^2_{L^\infty(\Omega)} \|\vp\|_{L^2(\Omega)}^2 
    + 2\|(q-q_k)\|^2_{L^\frac{n}{2}(\Omega)}\|\vp\|_{L^\frac{2n}{n-2}(\Omega)}^2  \\
    &\leq 
    2\|q_k\|^2_{L^\infty(\Omega)} \|\vp\|_{L^2(\Omega)}^2 
    + 2\|q-q_k\|^2_{L^\frac{n}{2}(\Omega)}\|\nabla \vp\|_{L^2(\Omega)}^2. 
\end{align*}
Using this with estimate \eqref{eq_split} and picking  $\|q-q_k\|_{L^\frac{n}{2}}$
to be small in a suitable way gives
\begin{align*}
    \int_\Omega q \vp^2 \leq \varepsilon\| \nabla \vp\|_{L^2(\Omega)}^2 + C_\varepsilon \| \vp \|_{L^2(\Omega)}^2.
\end{align*}
\end{proof}

\section{Appendix B. Besov Spaces } \label{sec:app2}
Here we review some basic definitions and  properties of the Besov spaces
that we use.
The main reason for considering Besov spaces is that they give  a more precise meaning to the 
restriction $u|_{\p \Omega}$, when $u$ is a member of the  Sobolev space $W^{s,p}(\Omega)$.
The main reference for this section is \cite{Tr}.

\medskip
\noindent
We use the following definitions of the spaces $B_{pp}^s(\Omega)$. We split
$s \in \R$, as $s = [s]+ \{s\}$, where $[s]$ is an integer
and $0\leq \{s\} <1$.
The space $B_{pp}^s(\R^n)$, with $1<p<\infty$ and $0 < s \notin \N$,
consists of those functions in $f \in L^p(\R^n)$, for which the norm
\begin{small}
\begin{align*}
    \|f\|_{B_{pp}^s(\R^n)} :=  \|f\|_ {W^{[s],p}(\R^n)} 
                           +  \sum_{|\alpha| = [s]} \left( \int_{\R^n}
        \frac{\|D^\alpha f(\cdot + y) - D^\alpha f \|^p_{L^p(\R^n)}}{|y|^{n + p\{s\}}}
 \, dy \right)^{1/p}, 
\end{align*}
\end{small}
is finite, where $\alpha:=(\alpha_1,\dots,\alpha_n)$ is a multi-index 
and  $D^\alpha := \p^{\alpha_1}_{x_1}\dots\p^{\alpha_n}_{x_n}$.
For a bounded domain $\Omega \subset \R^n$, we define the space  
$B_{pp}^s(\Omega)$ as the set
\[
    B_{pp}^s(\Omega) := \{ f \in L^p(\Omega) \colon \exists g \in B_{pp}^s(\R^n) \text{ s.t. } g|_{\Omega} = f\},
\]
equipped with the norm 
\[
    \|f\|_{B_{pp}^s(\Omega)} :=  
    \inf \{ \|g\|_{B_{pp}^s(\R^n)} \colon g \in B_{pp}^s(\R^n) \text{ s.t. } g|_{\Omega} = f\}.
\]
Our main interest in the spaces $B^s_{pp}(\Omega)$ is that they provide natural trace
spaces for the spaces $W^{2,p}(\Omega)$. 
The Dirichlet trace operator $\gamma$, may be defined by
\[
    \gamma u := u|_{\p\Omega}. %,\p_{\nu}u|_{\p \Omega},\dots,\p_{\nu}^{m-1}u|_{\p \Omega}),
\]
for $u\in C^\infty(\Omega)$. 
One can then extend this map to a continuous map $\gamma\colon W^{2,p}(\Omega) \to  B^{2-1/p}_{pp}(\p \Omega)$, 
so that\footnote{
    Firstly we have that $W^{2,p}(\Omega) = F^{2}_{p2}(\Omega)$, where $F^{2}_{p2}(\Omega)$
    is a Triebel space. See \cite{Tr}, 
    section 3.4.2, p. 208. By Theorem 3.3.3 in \cite{Tr}, we have on the other hand
    that
    \[
        \gamma:F^{2}_{p2}(\Omega) \to   B^{2-1/p}_{pp}(\p \Omega),
    \]
    is continuous}
\[
    \|\gamma u \|_{B_{pp}^{2-1/p}(\p\Omega)} \leq C  \| u \|_{W^{2,p}(\Omega)}.
\]
We will use the notation $u|_{\p \Omega} := \gamma u$.
The Dirichlet trace operator $\gamma$ has moreover a right inverse $E$, 
$E\colon  B^{2-1/p}_{pp}(\p \Omega) \to W^{2,p}(\Omega)$,
for which $\gamma E f = f$ and  
\[
   \| E f \|_{W^{2,p}(\Omega)} \leq C \| f \|_{B_{pp}^{2-1/p}(\p\Omega)},
\]
See section 3.3.3. in \cite{Tr}.

Likewise one can define the Neumann trace operator $\tilde{\gamma}$ given by
\[
    \tilde{\gamma} u := \p_{\nu} u|_{\p\Omega}  %,\dots,\p_{\nu}^{2m-1}u|_{\p \Omega}),
\]
for $u\in C^\infty(\Omega)$ and 
where $\nu$ is the outer unit normal vector to $\p \Omega$. The Neumann trace operator
\[
    \tilde{\gamma} \colon W^{2,p}(\Omega) \to  B^{1-1/p}_{pp}(\p \Omega),
\]
is bounded and linear, which follows similarly as the continuity for the
Dirichlet trace operator.
For $u \in W^{2,p}(\Omega)$, we will use for $\tilde \gamma u$  the notation
$\p_\nu u|_{\p\Omega}$.

\section*{Acknowledgements}  
The author would like to thank Esa Vesalainen, Katya Krupchyk and Lassi Päivärinta 
for some very useful discussions and suggestions.

%\bibliographystyle{plain}
%\bibliography{ref}

\end{document}